\documentclass[11pt,a4paper]{article}

\usepackage{inputenc}
\usepackage{amsmath}
\usepackage{bm}
\usepackage{bbold}
\usepackage{amsthm}
\usepackage{enumerate}

\usepackage{hyperref}

\title{A Multidimensional Tropical Optimisation Problem with a Nonlinear Objective Function and Linear Constraints}

\author{Nikolai Krivulin\thanks{Faculty of Mathematics and Mechanics, Saint Petersburg State University, 28 Universitetsky Ave., Saint Petersburg, 198504, Russia, 
nkk@math.spbu.ru.}\thanks{The work was supported in part by the Russian Foundation for Humanities under Grant \#13-02-00338.}
}

\date{}

\newtheorem{theorem}{Theorem}
\newtheorem{lemma}[theorem]{Lemma}
\newtheorem{corollary}[theorem]{Corollary}
\newtheorem{proposition}{Proposition}

\begin{document}

\maketitle

\begin{abstract}
We examine a multidimensional optimisation problem in the tropical mathematics setting. The problem involves the minimisation of a nonlinear function defined on a finite-dimensional semimodule over an idempotent semifield subject to linear inequality constraints. We start with an overview of known tropical optimisation problems with linear and nonlinear objective functions. A short introduction to tropical algebra is provided to offer a formal framework for solving the problem under study. As a preliminary result, a solution to a linear inequality with an arbitrary matrix is presented. We describe an example optimisation problem drawn from project scheduling and then offer a general representation of the problem. To solve the problem, we introduce an additional variable and reduce the problem to the solving of a linear inequality, in which the variable plays the role of a parameter. A necessary and sufficient condition for the inequality to hold is used to evaluate the parameter, whereas the solution to the inequality is considered a solution to the problem. Based on this approach, a complete direct solution in a compact vector form is derived for the optimisation problem under fairly general conditions. Numerical and graphical examples for two-dimensional problems are given to illustrate the obtained results.
\\

\\

\end{abstract}

\section{Introduction}

As an applied mathematical theory and methods based on the notion of idempotent semirings, tropical (idempotent) mathematics, has its origin in a few works, including \cite{Pandit1961Anew,CuninghameGreen1962Describing,Vorobjev1963Theextremal,Romanovskii1964Asymptotic}, all of which were inspired by problems from operations research. Over the last decades, a number of studies have achieved many theoretical results and developed various applications; these are reported in monographs \cite{Cuninghamegreen1979Minimax,Carre1979Graphs,Zimmermann1981Linear,Baccelli1993Synchronization,Kolokoltsov1997Idempotent,Golan2003Semirings,Heidergott2006Maxplus,Gondran2008Graphs,Krivulin2009Methods,Butkovic2010Maxlinear} and in a wide range of contributed papers.

One of the areas directly concerned with problems in operations research is tropical optimisation, which is an area that addresses the analysis and solution of optimisation problems formulated in the tropical mathematics setting. A minimax earliness problem in machine scheduling that was previously studied \cite{Cuninghamegreen1976Projections} is an early instance of the real-world problems that can be represented and solved within the framework of tropical mathematics. Further examples include multidimensional optimisation problems in location analysis, transportation networks, decision making, and discrete event systems.

Optimisation problems usually take the form of minimising (maximising) functions that are defined on finite-dimensional semimodules over idempotent semifields and may have additional equality and inequality constraints imposed on the feasible solutions. Some problems involve both an objective function and constraints that are linear in the tropical mathematics sense. These tropical linear optimisation problems, which appear to be formal analogues of those in conventional linear programming, were apparently first considered in \cite{Hoffman1963Onabstract,Superville1978Various,Zimmermann1981Linear}. 

In another important class of problems that dates back in the literature to the works \cite{Cuninghamegreen1976Projections,Cuninghamegreen1979Minimax,Zimmermann1981Linear}, the constraints are linear, but the objective function is not. This class actually contains problems with functions that involve a conjugate transposition operator, which does not preserve tropical linearity. 

A number of techniques have been developed to solve tropical optimisation problems. Some problems, such as those in the early works \cite{Hoffman1963Onabstract,Cuninghamegreen1976Projections,Superville1978Various,Zimmermann1981Linear}, can be completely solved, and the solution is obtained in an explicit vector form in terms of a general semiring. In contrast, the existing solutions given for other problems concentrate on a particular semifield and are obtained in the form of iterative algorithms that give particular solutions (if there are any) or indicate that there is no solution (see, for instance, \cite{Butkovic1984Onproperties,Zimmermann1984Some}).

In this paper, we examine a multidimensional optimisation problem that appears in project management when an optimal schedule is constructed to minimise the flow time of activities in a project under various activity precedence constraints.

We formulate the problem in terms of a general semimodule over an idempotent linearly ordered radicable semifield. Given two arbitrary matrices, the problem is to minimise a nonlinear objective function that is defined by one of the matrices through a multiplicative conjugate transpose and is subject to linear inequality constraints that are provided by the other matrix. The problem presents a sufficient extension of that considered in \cite{Krivulin2012Atropical}; in the previous study, an irreducibility condition was imposed on the matrices, and the inequality constraints were less general. Solutions to an unconstrained version of the problem, in addition to related applications in location analysis and stochastic discrete event systems, are presented in \cite{Cuninghamegreen1979Minimax,Krivulin2005Evaluation,Krivulin2006Eigenvalues,Krivulin2009Methods,Krivulin2011Analgebraic,Krivulin2011Anextremal}.

To solve the problem, we implement and further develop the results offered in \cite{Krivulin2006Solution,Krivulin2006Eigenvalues,Krivulin2009Methods}, including solutions to linear inequalities and extremal properties of the eigenvalues of matrices. We follow the approach proposed in \cite{Krivulin2012Atropical}, which introduced an additional unknown variable and then reduced the problem to the solving of a linear inequality, in which the new variable plays the role of a parameter. A necessary and sufficient condition for the inequality to hold is used to evaluate the parameter, whereas the solution of the inequality is taken as a solution to the initial problem. As a result, under fairly general conditions, a complete direct solution to the problem is derived in a compact vector form, which is suitable to both further analysis and applications.

The rest of the paper is as follows. We start with an overview of known tropical optimisation problems with both linear and nonlinear objective functions and linear constraints. We discuss related solution methods and indicate the areas of application. A short introduction to tropical algebra is then given to provide a formal framework in terms of a general idempotent semifield for the solving of the optimisation problem under study. The introduction includes graphical illustrations that are intended to clarify the main definitions and basic facts using the classical semifield $\mathbb{R}_{\max,+}$ as an example.

Furthermore, a complete explicit solution for a linear inequality is obtained as an important prerequisite. The solution extends the known results for the inequality with an irreducible matrix to the case of irreducible matrices. Finally, we offer an example optimisation problem drawn from project scheduling and provide a formal description of the problem. We exploit the above solution for inequalities to obtain a complete solution to the problem. Numerical examples for the solving of two-dimensional problems with graphical illustrations are also included.

\section{Tropical optimisation problems} 

Multidimensional optimisation problems form a significant research domain in tropical mathematics, dating back to the works \cite{Hoffman1963Onabstract,Cuninghamegreen1976Projections}. These problems arise in real-world applications in various areas. Specifically, \cite{Cuninghamegreen1976Projections,Zimmermann1981Linear,Zimmermann1984Some,Zimmermann2006Interval,Butkovic2009Introduction,Butkovic2009Onsome,Aminu2012Nonlinear} provide tropical solutions to problems in project scheduling. Furthermore, \cite{Cuninghamegreen1991Minimax,Krivulin2011Analgebraic,Krivulin2011Anextremal,Krivulin2012Anew} solve problems in location analysis. Applications have been  developed in transportation networks \cite{Zimmermann1981Linear,Zimmermann2006Interval}, decision making \cite{Elsner2004Maxalgebra}, and discrete event systems \cite{Krivulin2005Evaluation}.

Below, we provide a short overview of one class of optimisation problems and briefly discuss their solution methods. The problems are formulated to minimise functions defined on vectors of a semimodule over an idempotent linearly ordered radicable semifield. We consider both unconstrained and constrained problems, in which the constraints have the form of linear vector equations and inequalities.

We also present optimisation problems with both linear and nonlinear objective functions. However, only those nonlinear functions that are composed of a conjugate transposition operator are under study. As shown in our overview, a substantial part of the problems known in the literature directly falls in this category or can be placed in it after some equivalent transformations.

In the description of the problems, we use the symbols $\bm{A}$, $\bm{B}$, and $\bm{C}$ for given matrices, $\bm{b}$, $\bm{d}$, $\bm{p}$, and $\bm{q}$ for given vectors, and $\bm{x}$ for the unknown vector. The matrix and vector operations are considered to be defined in terms of an idempotent semifield. The minus sign in the exponent denotes the multiplicative conjugate transpose.

Specifically, for the real semifield $\mathbb{R}_{\max,+}$, matrix addition and matrix and scalar multiplications follow the conventional rules, in which maximum and addition play the roles of scalar addition and multiplication, respectively. The multiplicative conjugate transpose of a column vector is a row vector that is obtained by replacing the former one with its usual opposite and transposition. Further details on the notation used here and throughout the paper and the related definitions are included in the next section.

\subsection{Linear objective functions}

One of the long-known and well-studied tropical optimisation problems is a direct tropical analogue of linear programming problems:
\begin{equation*}
\begin{aligned}
&
\text{minimise}
&&
\bm{p}^{T}\bm{x},
\\
&
\text{subject to}
&&
\bm{A}\bm{x}
\geq
\bm{d}.
\end{aligned}
\end{equation*}

Complete closed-form solutions to the problem under various algebraic assumptions, in addition to the related duality results, are provided in a number of publications. Specifically, \cite{Hoffman1963Onabstract,Zimmermann1981Linear} examined the problem for the case of general semirings; the former study offers a solution based on an abstract extension of the conventional linear programming duality, and the second study suggested a residual-based solution technique. Furthermore, \cite{Superville1978Various} developed the results in \cite{Hoffman1963Onabstract} to concentrate on the idempotent semifield $\mathbb{R}_{\max,+}$, and \cite{Gavalec2012Duality} provided a solution for $\mathbb{R}_{\max,+}$ and $\mathbb{R}_{\max,\times}$ in the context of the theory of max-separable functions.

A tropical optimisation problem in $\mathbb{R}_{\max,+}$ with additional constraints, which can be written in the form
\begin{equation*}
\begin{aligned}
&
\text{minimise}
&&
\bm{p}^{T}\bm{x},
\\
&
\text{subject to}
&&
\bm{A}\bm{x}
\leq
\bm{d},
\quad
\bm{C}\bm{x}
\geq
\bm{b},
\\
&&&
\bm{g}
\leq
\bm{x}
\leq
\bm{h},
\end{aligned}
\end{equation*}
was considered in \cite{Zimmermann1984Onmaxseparable,Zimmermann1992Optimization,Zimmermann2003Disjunctive,Zimmermann2006Interval} within the framework of max-separable functions. Under general conditions, explicit solutions to the problem are given in conventional terms rather than in a closed form of tropical vector algebra.

The solutions to another problem in $\mathbb{R}_{\max,+}$ with two-sided equality constraints,
\begin{equation*}
\begin{aligned}
&
\text{minimise}
&&
\bm{p}^{T}\bm{x},
\\
&
\text{subject to}
&&
\bm{A}\bm{x}\oplus\bm{b}
=
\bm{C}\bm{x}\oplus\bm{d},
\end{aligned}
\end{equation*}
were obtained in \cite{Butkovic1984Onproperties,Butkovic2010Maxlinear,Butkovic2009Introduction,Aminu2012Nonlinear}. In particular, \cite{Butkovic2009Introduction} presented a pseudo-polynomial algorithm to evaluate a solution or to indicate that no solutions exist. The algorithm uses an alternating method developed by \cite{Cuninghamegreen2003Theequation} to replace the equality constraint with two opposite inequalities, which are solved iteratively to provide progressively better bounds for a solution. A heuristic approach was suggested in \cite{Aminu2012Nonlinear} to obtain an approximate solution using local search techniques combined with iterative procedures that solve low-dimensional problems with one and two variables.

\subsection{Nonlinear objective functions}

We start with a particular problem that can be written in the form
\begin{equation*}
\begin{aligned}
&
\text{minimise}
&&
(\bm{A}\bm{x})^{-}\bm{d},
\\
&
\text{subject to}
&&
\bm{A}\bm{x}
\leq
\bm{d}.
\end{aligned}
\end{equation*}

The problem is formulated to find a vector $\bm{x}$ that provides the best underestimating approximation to the vector $\bm{d}$ by vectors $\bm{A}\bm{x}$ in the sense of the Chebyshev distance, which is defined as the maximum absolute difference between the corresponding coordinates of the vectors. A complete explicit solution is given in \cite{Cuninghamegreen1976Projections} as an application of an abstract theory of linear operators over $\mathbb{R}_{\max,+}$. A similar solution was also suggested by \cite{Zimmermann1981Linear}.

Another two problems are initially represented in a different manner but can be written in a form with a tropical nonlinear objective function and linear constraints. A problem in which a Chebyshev distance function is minimised under some constraints was defined in $\mathbb{R}_{\max,+}$ and solved with a polynomial-time threshold-type algorithm in \cite{Zimmermann1984Some}. In fact, using the same algebraic manipulations described in \cite{Krivulin2011Analgebraic,Krivulin2012Anew}, the problem can be rearranged to take the form
\begin{equation*}
\begin{aligned}
&
\text{minimise}
&&
(\bm{A}\bm{x})^{-}\bm{p}\oplus\bm{q}^{-}\bm{A}\bm{x},
\\
&
\text{subject to}
&&
\bm{g}
\leq
\bm{x}
\leq
\bm{h}.
\end{aligned}
\end{equation*}

A problem that minimises the range norm, which is defined as the maximum deviation between the coordinates of a vector, was considered in \cite{Butkovic2009Onsome}, and an explicit solution was suggested in a combined framework that involves both the semifield $\mathbb{R}_{\max,+}$ and its dual $\mathbb{R}_{\min,+}$. The problem can actually be written only in terms of $\mathbb{R}_{\max,+}$ as
\begin{equation*}
\begin{aligned}
&
\text{minimise}
&&
\mathbb{1}^{T}\bm{A}\bm{x}(\bm{A}\bm{x})^{-}\mathbb{1},
\end{aligned}
\end{equation*}
where $\mathbb{1}=(\mathbb{1},\ldots,\mathbb{1})^{T}$ is a vector of ones (in terms of the semifield $\mathbb{R}_{\max,+}$).

Based on a connection between the solutions of two-sided equality constraints and a mean payoff game established by \cite{Akian2012Tropical}, an iterative computational scheme was developed in \cite{Gaubert2012Tropical} to find a solution in $\mathbb{R}_{\max,+}$ to a problem with a nonlinear objective function and a two-sided inequality constraint given by
\begin{equation*}
\begin{aligned}
&
\text{minimise}
&&
\bm{p}^{T}\bm{x}(\bm{q}^{T}\bm{x})^{-1},
\\
&
\text{subject to}
&&
\bm{A}\bm{x}\oplus\bm{b}
\leq
\bm{C}\bm{x}\oplus\bm{d}.
\end{aligned}
\end{equation*}

We now consider problems that are formulated in terms of a general linear ordered radicable semifield and admit explicit solutions in a vector form. The analytic technique implemented to solve these problems is based on new results in tropical spectral theory and solutions to the linear tropical inequality developed in \cite{Krivulin2006Eigenvalues,Krivulin2006Solution,Krivulin2009Methods}.

First, a complete solution is given in \cite{Krivulin2012Acomplete} to the unconstrained problem
\begin{equation*}
\begin{aligned}
&
\text{minimise}
&&
\bm{x}^{-}\bm{A}\bm{x}\oplus\bm{x}^{-}\bm{p}\oplus\bm{q}^{-}\bm{x}.
\end{aligned}
\end{equation*}

Furthermore, \cite{Krivulin2012Anew} offers a direct explicit solution to another unconstrained problem written in the form
\begin{equation*}
\begin{aligned}
&
\text{minimise}
&&
(\bm{A}\bm{x})^{-}\bm{p}\oplus\bm{q}^{-}\bm{A}\bm{x}.
\end{aligned}
\end{equation*}

It has also been shown that two problems with linear inequality and equality constraints,
\begin{equation*}
\begin{aligned}
&
\text{minimise}
&&
\bm{x}^{-}\bm{p}\oplus\bm{q}^{-}\bm{x},
\\
&
\text{subject to}
&&
\bm{A}\bm{x}
\leq
\bm{x};
\end{aligned}
\qquad\qquad
\begin{aligned}
&
\text{minimise}
&&
\bm{x}^{-}\bm{p}\oplus\bm{q}^{-}\bm{x},
\\
&
\text{subject to}
&&
\bm{A}\bm{x}
=
\bm{x};
\end{aligned}
\end{equation*}
can be reduced to the former unconstrained problem and can thus be solved to obtain explicit solutions.

In addition, a brief conference paper \cite{Krivulin2012Atropical} provides a complete direct solution to the problem
\begin{equation*}
\begin{aligned}
&
\text{minimise}
&&
\bm{x}^{-}\bm{A}\bm{x},
\\
&
\text{subject to}
&&
\bm{B}\bm{x}
\leq
\bm{x}
\end{aligned}
\end{equation*}
under the condition that at least one of the matrices $\bm{A}$ and $\bm{B}$ is irreducible.

In this paper, we examine in detail a more general version of the last problem with arbitrary matrices and additional constraints.

\section{Preliminary definitions and results}

The purpose of this section is to offer a brief introduction to tropical (idempotent) mathematics to provide an appropriate theoretical framework for the subsequent analysis and solution of tropical optimisation problems.

Both concise introductions to and comprehensive presentations of the theory and methods are suggested in a range of published works \cite{Vorobjev1967Extremal,Cuninghamegreen1979Minimax,Cuninghamegreen1994Minimax,Carre1979Graphs,Zimmermann1981Linear,Baccelli1993Synchronization,Kolokoltsov1997Idempotent,Golan2003Semirings,Heidergott2006Maxplus,Litvinov2007Themaslov,Gondran2008Graphs,Butkovic2010Maxlinear}, some of which vary the notation system adopted and the form of the results presented.

Below, we give an overview of the definitions, notation, and background results, which mainly follow \cite{Krivulin2006Solution,Krivulin2006Eigenvalues,Krivulin2009Methods}, to form a basis for the complete solution to the problem under study in an explicit vector form and to provide clear geometric illustrations. Additional details and an extensive bibliography can be found in the publications listed above.

\subsection{Idempotent semifield}

Let $\mathbb{X}$ be a set equipped with two operations (addition $\oplus$ and multiplication $\otimes$) and their related neutral elements (null $\mathbb{0}$ and identity $\mathbb{1}$). Both operations are assumed to be associative and commutative, and multiplication is assumed to be distributive over addition. Furthermore, addition is idempotent and thus provides $x\oplus x=x$ for all $x\in\mathbb{X}$. Multiplication is invertible, which means that there exist a multiplicative inverse $x^{-1}$ for any $x\in\mathbb{X}_{+}$, where $\mathbb{X}_{+}=\mathbb{X}\setminus\{\mathbb{0}\}$. Endowed with these properties, the algebraic structure $\langle\mathbb{X},\mathbb{0},\mathbb{1},\oplus,\otimes\rangle$ is commonly referred to in the literature as the idempotent semifield.

Due to the associativity of multiplication, the power notation with integer exponents is introduced in the usual way. For any $x\in\mathbb{X}_{+}$ and an integer $p\geq1$, $x^{0}=\mathbb{1}$, $\mathbb{0}^{p}=\mathbb{0}$, $x^{p}=x^{p-1}\otimes x=x\otimes x^{p-1}$, and $x^{-p}=(x^{-1})^{p}$. Moreover, we assume that the integer power extends to rational exponents, which makes the semifield radicable. 

For the sake of brevity, we do not use the multiplication sign $\otimes$ for here on, whereas the exponents are considered in the sense of the above notation.

A partial order is induced on $\mathbb{X}$ by idempotent addition such that $x\leq y$ if and only if $x\oplus y=y$. This definition implies an extremal property of addition in the form of inequalities $x\leq x\oplus y$ and $y\leq x\oplus y$. Moreover, according to the order, both addition and multiplication are monotone, which implies that the inequalities $x\leq u$ and $y\leq v$ result in $x\oplus y\leq u\oplus v$ and $xy\leq uv$, respectively. 

Furthermore, the partial order is assumed to extend to a total order, which makes the semifield linearly ordered. Throughout the paper, we interpret the relation symbols and the problem formulations in terms of this linear order.

Instances of the linearly ordered radicable idempotent semifield under study include $\mathbb{R}_{\max,+}=\langle\mathbb{R}\cup\{-\infty\},-\infty,0,\max,+\rangle$, $\mathbb{R}_{\min,+}=\langle\mathbb{R}\cup\{+\infty\},+\infty,0,\min,+\rangle$, $\mathbb{R}_{\max,\times}=\langle\mathbb{R}_{+}\cup\{0\},0,1,\max,\times\rangle$, and $\mathbb{R}_{\min,\times}=\langle\mathbb{R}_{+}\cup\{+\infty\},+\infty,1,\min,\times\rangle$, where $\mathbb{R}$ is the set of real numbers and $\mathbb{R}_{+}=\{x\in\mathbb{R}|x>0\}$. 

For example, the semifield $\mathbb{R}_{\max,+}$ has null $\mathbb{0}=-\infty$ and identity $\mathbb{1}=0$. An inverse $x^{-1}$ is defined for each $x\in\mathbb{R}$ and is equal to $-x$ in conventional algebra. For any $x,y\in\mathbb{R}$, the power $x^{y}$ coincides with the arithmetic product $xy$. The order induced by idempotent addition corresponds to the natural linear order on $\mathbb{R}$.

The semifield $\mathbb{R}_{\min,\times}$ is equipped with $\mathbb{0}=+\infty$ and $\mathbb{1}=1$. The inverse and power notations exhibit the standard interpretation. The relation $\leq$ defines a reverse order to the linear order on $\mathbb{R}$. 

Note that all of the above mentioned semifields are isomorphic to each other.

\subsection{Idempotent semimodule}

Let $\mathbb{X}^{n}$ be the Cartesian power with elements that are taken to be column vectors. There is a vector $\mathbb{0}\in\mathbb{X}^{n}$ that contains all zero components. Any vector without zero components is called regular. The set of all regular vectors in $\mathbb{X}^{n}$ is represented by $\mathbb{X}_{+}^{n}$.

For any vectors $\bm{a}=(a_{i})$ and $\bm{b}=(b_{i})$ from $\mathbb{X}^{n}$, and a scalar $x\in\mathbb{X}$, vector addition and scalar multiplication are defined component-wise as
$$
\{\bm{a}\oplus\bm{b}\}_{i}
=
a_{i}\oplus b_{i},
\qquad
\{x\bm{a}\}_{i}
=
xa_{i}.
$$

Endowed with these operations, the Cartesian power $\mathbb{X}^{n}$ forms an idempotent semimodule over $\mathbb{X}$.

The extremal property of scalar addition extends to the component-wise vector inequalities $\bm{a}\leq\bm{a}\oplus\bm{b}$ and $\bm{b}\leq\bm{a}\oplus\bm{b}$. The monotonicity of the scalar operations implies that vector addition and scalar multiplication are monotone in each argument.

To provide a clear visualisation of the key phenomena inherent in idempotent semimodules, we use graphical illustrations using the semimodule $\mathbb{R}_{\max,+}^{2}$. Any vector in $\mathbb{R}_{\max,+}^{2}$ is shown on the plane with a Cartesian coordinate system, which can be supplemented by two artificial points if needed. These points are placed before the left end of the horizontal axis and below the bottom of the vertical axis to schematically indicate the value $\mathbb{0}=-\infty$ required to represent zero coordinates of irregular vectors.

A geometric interpretation of the operations is given in Figures~\ref{F-VA} and \ref{F-SM}. The idempotent addition of two vectors with a common base point at the origin of the coordinate system follows a ``rectangle rule''. According to the rule, the sum is obtained as the upper right vertex of an upright finite (Figure~\ref{F-VA}, left) or infinite (Figure~\ref{F-VA}, right) rectangle that has vertices at the end points of the vectors that are being added.
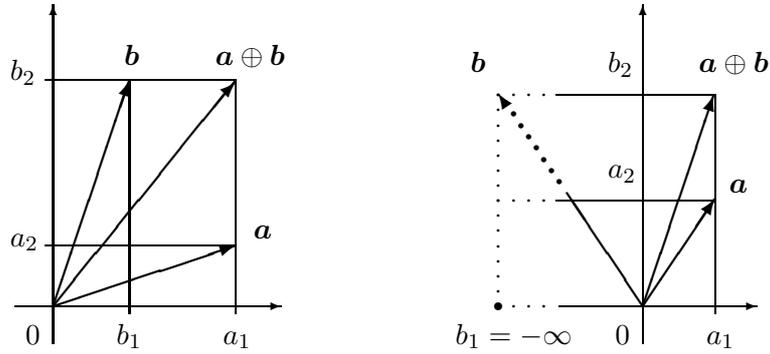
\begin{figure*}
\setlength{\unitlength}{1mm}
\begin{center}
\begin{picture}(35,45)

\put(0,5){\vector(1,0){35}}
\put(5,0){\vector(0,1){45}}

\put(5,5){\thicklines\vector(1,3){10}}
\put(15,35){\line(0,-1){31}}

\put(5,5){\thicklines\vector(3,1){24}}
\put(29,13){\line(-1,0){25}}

\put(5,5){\thicklines\line(4,5){24}}
\put(26,31.75){\thicklines\vector(1,1){3}}

\put(29,35){\line(-1,0){25}}

\put(29,35){\line(0,-1){31}}

\put(1,0){$0$}

\put(13,0){$b_{1}$}
\put(27,0){$a_{1}$}

\put(-1,13){$a_{2}$}
\put(-1,35){$b_{2}$}

\put(14,37){$\bm{b}$}

\put(31,14){$\bm{a}$}

\put(26,37){$\bm{a}\oplus\bm{b}$}

\end{picture}
\hspace{20\unitlength}
\begin{picture}(40,45)

\put(14,5){\vector(1,0){26}}
\put(25,0){\vector(0,1){45}}

\put(6,5){\circle*{1}}
\multiput(6,5)(2,0){4}{\circle*{.5}}

\put(25,5){\thicklines\line(-2,3){10}}
\multiput(14,21.5)(-1,1.5){8}{\circle*{.75}}

\put(8,30.25){\thicklines\vector(-2,3){2}}

\multiput(6,5)(0,2){15}{\circle*{.5}}

\multiput(6,33)(2,0){4}{\circle*{.5}}

\put(25,5){\thicklines\vector(2,3){9.5}}
\put(34.5,19){\line(-1,0){21}}
\multiput(6,19)(2,0){4}{\circle*{.5}}

\put(25,5){\thicklines\vector(1,3){9.3}}

\put(34.5,33){\line(-1,0){21}}

\put(34.5,33){\line(0,-1){29}}

\put(0,0){$b_{1}=-\infty$}

\put(21,0){$0$}

\put(33,0){$a_{1}$}

\put(20,22){$a_{2}$}
\put(20,36){$b_{2}$}

\put(2,36){$\bm{b}$}

\put(36,20){$\bm{a}$}

\put(32,36){$\bm{a}\oplus\bm{b}$}

\end{picture}
\end{center}
\caption{Vector addition of regular vectors (left) and of regular and irregular vectors (right).}
\label{F-VA}
\end{figure*}

The scalar multiplication of a vector by a nonzero scalar results in the moving of its end point along $45^{\circ}$ line to the coordinate axes for regular vectors (Figure~\ref{F-SM}, left) or along an artificial vertical (horizontal) line for irregular vectors (Figure~\ref{F-SM}, right).
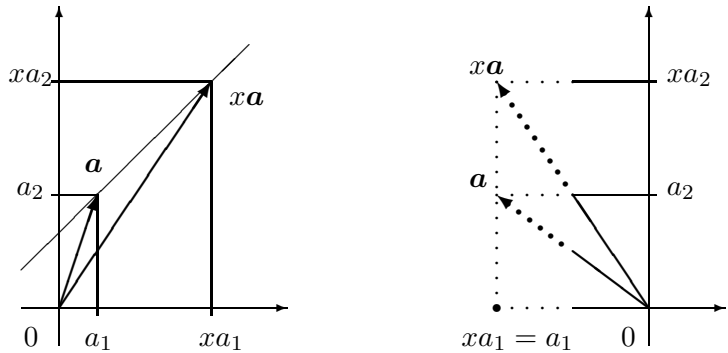
\begin{figure*}
\setlength{\unitlength}{1mm}
\begin{center}
\begin{picture}(35,45)

\put(0,5){\vector(1,0){35}}
\put(5,0){\vector(0,1){45}}

\put(5,5){\thicklines\vector(1,3){5}}
\put(10,20){\line(0,-1){16}}

\put(10,20){\line(-1,0){6}}

\put(5,5){\thicklines\vector(2,3){20}}
\put(25,35){\line(0,-1){31}}

\put(25,35){\line(-1,0){21}}

\put(0,10){\line(1,1){30}}

\put(0,0){$0$}
\put(8,23){$\bm{a}$}
\put(27,32){$x\bm{a}$}

\put(-1,20){$a_{2}$}
\put(-2,35){$xa_{2}$}

\put(8,0){$a_{1}$}
\put(23,0){$xa_{1}$}

\end{picture}
\hspace{20\unitlength}
\begin{picture}(35,45)

\put(15,5){\vector(1,0){20}}
\put(25,0){\vector(0,1){45}}

\put(5,5){\circle*{1}}
\multiput(5,5)(2,0){5}{\circle*{.5}}

\put(25,5){\thicklines\line(-2,3){10}}
\multiput(14,21.5)(-1,1.5){9}{\circle*{.75}}
\put(7,32){\thicklines\vector(-2,3){2}}

\multiput(5,5)(0,2){16}{\circle*{.5}}

\multiput(5,35)(2,0){6}{\circle*{.5}}

\put(25,5){\thicklines\line(-4,3){10}}
\multiput(13.5,13.5)(-1.4,1.05){6}{\circle*{.75}}
\put(7,18.5){\thicklines\vector(-4,3){2}}

\put(26,20){\line(-1,0){11}}
\multiput(5,20)(2,0){6}{\circle*{.5}}

\put(26,35){\line(-1,0){11}}

\put(0,0){$xa_{1}=a_{1}$}

\put(21,0){$0$}

\put(27,20){$a_{2}$}
\put(27,35){$xa_{2}$}

\put(1,36){$x\bm{a}$}
\put(1,21){$\bm{a}$}

\end{picture}
\end{center}
\caption{Scalar multiplication of a regular vector (left) and of an irregular vector (right).}
\label{F-SM}
\end{figure*}

As usual, a vector $\bm{b}$ is said to be linearly dependent on vectors $\bm{a}_{1},\ldots,\bm{a}_{m}$ if $\bm{b}=x_{1}\bm{a}_{1}\oplus\cdots\oplus x_{m}\bm{a}_{m}$ for some scalars $x_{1},\ldots,x_{m}$. Specifically, two vectors $\bm{a}$ and $\bm{b}$ are collinear if there is a scalar $x$ such that $\bm{b}=x\bm{a}$.

Given vectors $\bm{a}_{1},\ldots,\bm{a}_{m}$, the linear combinations $x_{1}\bm{a}_{1}\oplus\cdots\oplus x_{m}\bm{a}_{m}$ for all $x_{1},\ldots,x_{m}$ form a linear span of the vectors that constitutes an idempotent subsemimodule. Graphical examples of the linear span in $\mathbb{R}_{\max,+}^{2}$ are shown in Figure~\ref{F-LS} as areas that are surrounded by hatched borders.

The linear span of two regular vectors is a strip bounded by the lines drawn through the end points of the vectors (Figure~\ref{F-LS}, left). When one of the vectors is irregular, the strip extends to take the form of a half-plane (Figure~\ref{F-LS}, right).
\begin{figure*}
\setlength{\unitlength}{1mm}
\begin{center}
\begin{picture}(45,40)

\put(0,5){\vector(1,0){45}}
\put(10,0){\vector(0,1){45}}

\put(10,5){\thicklines\vector(-1,2){6}}
\put(10,5){\thicklines\vector(1,3){9.3}}

\put(10,5){\thicklines\vector(4,1){14}}
\put(10,5){\thicklines\vector(2,1){21}}

\put(0.5,14){\thicklines\line(1,1){25}}
\multiput(1.5,15)(1,1){24}{\line(1,0){1}}

\put(15.5,0){\thicklines\line(1,1){25}}
\multiput(16.5,1)(1,1){24}{\line(0,1){1}}

\put(31,33){\line(-1,0){21}}
\put(31,33){\line(0,-1){28}}

\put(10,5){\thicklines\vector(3,4){21}}

\put(5,0){$0$}

\put(1,21){$\bm{a}_{2}$}
\put(13,36){$x_{2}\bm{a}_{2}$}

\put(26,7){$\bm{a}_{1}$}
\put(34,15){$x_{1}\bm{a}_{1}$}

\put(26,36){$x_{1}\bm{a}_{1}\oplus x_{2}\bm{a}_{2}$}

\end{picture}
\hspace{15\unitlength}
\begin{picture}(50,45)

\put(15,5){\vector(1,0){35}}
\put(25,0){\vector(0,1){45}}

\put(6,5){\circle*{1}}
\multiput(6,5)(2.2,0){4}{\circle*{.5}}

\put(25,5){\thicklines\line(-2,3){10}}
\multiput(14,21.5)(-1,1.5){7}{\circle*{.75}}
\put(8,30){\thicklines\vector(-2,3){2}}

\multiput(6,5)(0,2){15}{\circle*{.75}}

\multiput(6,33)(2.2,0){4}{\circle*{.5}}

\put(25,5){\thicklines\line(-4,3){10}}
\multiput(13.5,13.5)(-1.4,1.05){4}{\circle*{.75}}
\put(8,17.5){\thicklines\vector(-4,3){2}}


\put(39,33){\line(-1,0){24}}

\put(39,33){\line(0,-1){28}}


\put(25,5){\thicklines\vector(3,-1){7}}

\put(25,5){\thicklines\vector(3,1){14}}

\put(25,5){\thicklines\vector(1,2){14}}

\put(29.5,0){\thicklines\line(1,1){16}}
\multiput(30.5,1)(1,1){15}{\line(0,1){1}}


\put(21,0){$0$}


\put(1,20){$\bm{a}_{2}$}
\put(0,36){$x_{2}\bm{a}_{2}$}

\put(33,0){$\bm{a}_{1}$}
\put(41,8){$x_{1}\bm{a}_{1}$}

\put(32,36){$x_{1}\bm{a}_{1}\oplus x_{2}\bm{a}_{2}$}

\end{picture}
\end{center}
\caption{Linear span of regular vectors (left) and of regular and irregular vectors (right).}
\label{F-LS}
\end{figure*}
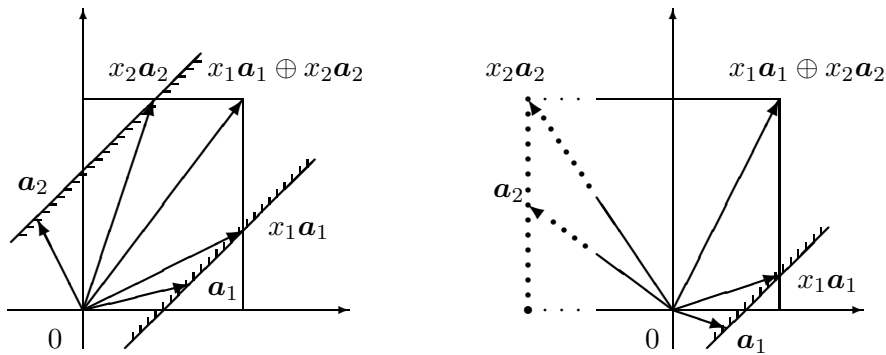

For any nonzero column vector $\bm{a}=(a_{i})$, we define its multiplicative conjugate transpose as a row vector $\bm{a}^{-}=(a_{i}^{-})$ with components $a_{i}^{-}=a_{i}^{-1}$ if $a_{i}>\mathbb{0}$, and $a_{i}^{-}=\mathbb{0}$ otherwise. If both vectors $\bm{a}$ and $\bm{b}$ are regular, then the component-wise inequality $\bm{a}\leq\bm{b}$ implies that $\bm{a}^{-}\geq\bm{b}^{-}$.

\subsection{Matrix algebra}

Consider matrices with entries from $\mathbb{X}$. For any conforming matrices $\bm{A}=(a_{ij})$, $\bm{B}=(b_{ij})$, and $\bm{C}=(c_{ij})$ and scalar $x\in\mathbb{X}$, matrix addition and multiplication and scalar multiplication follow the element-wise formulas
$$
\{\bm{A}\oplus\bm{B}\}_{ij}
=
a_{ij}
\oplus
b_{ij},
\qquad
\{\bm{A}\bm{C}\}_{ij}
=
\bigoplus_{k}a_{ik}c_{kj},
\qquad
\{x\bm{A}\}_{ij}
=
xa_{ij}.
$$

Specifically, the multiplication of a matrix $\bm{A}=(a_{ij})$ by a vector $\bm{x}=(x_{i})$ of order $n$ gives a vector with components
$$
\{\bm{A}\bm{x}\}_{i}
=
a_{i1}x_{1}\oplus\cdots\oplus a_{in}x_{n}.
$$

The matrix operations are element-wise monotone in each argument. If $\bm{A}$ and $\bm{B}$ are matrices of the same size, then $\bm{A}\leq\bm{A}\oplus\bm{B}$ and $\bm{B}\leq\bm{A}\oplus\bm{B}$.

A matrix that has only zero entries is the zero matrix represented by $\mathbb{0}$.

Consider square matrices of order $n$ over $\mathbb{X}$, and denote the set of matrices $\mathbb{X}^{n\times n}$. Any matrix with all off-diagonal entries equal to $\mathbb{0}$ is a diagonal matrix. A diagonal matrix with $\mathbb{1}$ along the diagonal is the identity matrix, which is denoted $\bm{I}$.

The matrix power is introduced in a regular manner. For any matrix $\bm{A}\in\mathbb{X}^{n\times n}$ and integer $p\geq1$, $\bm{A}^{0}=\bm{I}$ and $\bm{A}^{p}=\bm{A}^{p-1}\bm{A}=\bm{A}\bm{A}^{p-1}$.

If all of the entries of a matrix above or below the diagonal are zero, the matrix is triangular. A triangular matrix with zero diagonal entries is strictly triangular.

A matrix is reducible if simultaneous permutations of its rows and columns can transform it into a block-triangular normal form and irreducible otherwise. The lower triangular normal form of a matrix $\bm{A}$ is given by
\begin{equation}\label{E-MNF}
\bm{A}
=
\left(
	\begin{array}{cccc}
		\bm{A}_{11}	& \mathbb{0}	& \ldots	& \mathbb{0} \\
		\bm{A}_{21}	& \bm{A}_{22}	&					& \mathbb{0} \\
		\vdots			& \vdots			& \ddots	& \\
		\bm{A}_{s1}	& \bm{A}_{s2}	& \ldots	&	\bm{A}_{ss}
	\end{array}
\right),
\end{equation}
where $\bm{A}_{ii}$ is either irreducible or a zero matrix of order $n_{i}$ and $\bm{A}_{ij}$ is an arbitrary matrix of size $n_{i}\times n_{j}$ for all $i=1,\ldots,s$, $j<i$, and $n_{1}+\cdots+n_{s}=n$.

The trace of any matrix $\bm{A}=(a_{ij})$ is routinely defined as
$$
\mathop\mathrm{tr}\bm{A}
=
a_{11}\oplus\cdots\oplus a_{nn}.
$$

For any matrices $\bm{A}$ and $\bm{B}$, and scalar $x$, the following equalities hold:
$$
\mathop\mathrm{tr}(\bm{A}\oplus\bm{B})
=
\mathop\mathrm{tr}\bm{A}
\oplus
\mathop\mathrm{tr}\bm{B},
\qquad
\mathop\mathrm{tr}(\bm{A}\bm{B})
=
\mathop\mathrm{tr}(\bm{B}\bm{A}),
\qquad
\mathop\mathrm{tr}(x\bm{A})
=
x\mathop\mathrm{tr}(\bm{A}).
$$

The application of these properties leads, in particular, to a binomial identity for traces derived in \cite{Krivulin2012Atropical} in the form
\begin{equation}
\mathop\mathrm{tr}(\bm{A}\oplus\bm{B})^{m}
=
\bigoplus_{k=1}^{m}\mathop{\bigoplus\hspace{2.3em}}_{i_{1}+\cdots+i_{k}=m-k}\mathop\mathrm{tr}(\bm{A}\bm{B}^{i_{1}}\cdots\bm{A}\bm{B}^{i_{k}})
\oplus
\mathop\mathrm{tr}\bm{B}^{m},
\label{E-trABm}
\end{equation}
which is valid for any non-negative integer $m$.

As usual, a scalar $\lambda$ is an eigenvalue of a matrix $\bm{A}$ if there exists a nonzero vector $\bm{x}$ such that
$$
\bm{A}\bm{x}
=
\lambda\bm{x}.
$$

Every irreducible matrix has only one eigenvalue, whereas reducible matrices may possess several eigenvalues. The maximal eigenvalue (in the sense of the order on $\mathbb{X}$) is called the spectral radius of $\bm{A}$ and is directly calculated as
$$
\lambda
=
\bigoplus_{m=1}^{n}\mathop\mathrm{tr}\nolimits^{1/m}(\bm{A}^{m}).
$$

The spectral radius $\lambda$ of any matrix $\bm{A}$ offers a useful extremal property that was considered by \cite{Cuninghamegreen1979Minimax,Krivulin2005Evaluation}, which stated that
$$
\min\ \bm{x}^{-}\bm{A}\bm{x}
=
\lambda,
$$
where the minimum is taken over all regular vectors $\bm{x}$.

\section{Explicit solution to a linear inequality}

Explicit results to the optimisation problems that will be derived in the next section make use of direct solutions to the following problem. Given a matrix $\bm{A}\in\mathbb{X}^{n\times n}$ and vector $\bm{b}\in\mathbb{X}^{n}$, find all regular vectors $\bm{x}\in\mathbb{X}^{n}_{+}$ that satisfy the inequality
\begin{equation}
\bm{A}\bm{x}\oplus\bm{b}
\leq
\bm{x}.
\label{I-Axbx}
\end{equation}

The purpose of this section is to derive the solution to inequality \eqref{I-Axbx} in the form of an expression that describes a general vector in the solution set, and which is usually referred to as the general or complete solution.

An equation with the same terms as in \eqref{I-Axbx} is frequently referred to as the non-homogeneous Bellman equation. This equation has been examined in many works \cite{Carre1971Analgebra,Carre1979Graphs,Zimmermann1981Linear,Gondran1984Linear,Gondran2008Graphs,Mahr1984Iteration,Baccelli1993Synchronization,Kolokoltsov1997Idempotent}, whereas the solution to the inequality appears to have received less attention in the literature.

Complete solutions to the inequality with both irreducible and reducible matrices can be found in \cite{Krivulin2006Solution,Krivulin2009Methods}, where they were obtained by reducing the inequality to an equation. Below, we concentrate on regular solutions in the case of general reducible matrices and offer a direct proof to obtain an explicit solution in a slightly different and more compact form.

For any matrix $\bm{A}$ of order $n$, we define a function that is given by
$$
\mathop\mathrm{Tr}(\bm{A})
=
\mathop\mathrm{tr}\bm{A}\oplus\cdots\oplus\mathop\mathrm{tr}\bm{A}^{n},
$$
and a star operator that takes $\bm{A}$ to a matrix
$$
\bm{A}^{\ast}
=
\bm{I}\oplus\bm{A}\oplus\cdots\oplus\bm{A}^{n-1}.
$$

We use the function to reformulate the condition in a known result that was apparently first suggested in \cite{Carre1971Analgebra} and is referred to below as the Carr{\'e} inequality. The result states that any matrix $\bm{A}$ with $\mathop\mathrm{Tr}(\bm{A})\leq\mathbb{1}$ satisfies the inequality
$$
\bm{A}^{k}
\leq
\bm{A}^{\ast},
\qquad
k
\geq
0.
$$

For irreducible matrices, all regular solutions to \eqref{I-Axbx} can be described as follows (see \cite{Krivulin2006Solution,Krivulin2009Methods} for the proof and further details).
\begin{lemma}\label{L-IAxbx}
Let $\bm{x}$ be the complete regular solution to inequality \eqref{I-Axbx} with an irreducible matrix $\bm{A}$. Then, the following statements hold:
\begin{enumerate}
\item If $\mathop\mathrm{Tr}(\bm{A})\leq\mathbb{1}$, then $\bm{x}=\bm{A}^{\ast}\bm{u}$ for all regular vectors $\bm{u}\geq\bm{b}$.
\item If $\mathop\mathrm{Tr}(\bm{A})>\mathbb{1}$, then there is no regular solution.
\end{enumerate}
\end{lemma}

Figure~\ref{F-GSIAxbx} offers examples of solutions to inequalities that have a common matrix $\bm{A}=(\bm{a}_{1},\bm{a}_{2})$ and different vectors $\bm{b}$ in $\mathbb{R}_{\max,+}^{2}$. The set of solutions is indicated by semi-infinite regions with hatched borders. Together with the columns of the matrix $\bm{A}$, those of $\bm{A}^{\ast}=(\bm{a}_{1}^{\ast},\bm{a}_{2}^{\ast})$ are also presented.
\begin{figure}
\setlength{\unitlength}{1mm}
\begin{center}
\begin{picture}(50,45)

\put(0,15){\vector(1,0){50}}
\put(20,0){\vector(0,1){45}}

\put(20,15){\thicklines\vector(-1,0){10}}
\put(6,11){\thicklines\line(1,1){30}}
\multiput(26,31)(1,1){10}{\line(1,0){1}}

\put(20,15){\thicklines\vector(-1,-4){3}}

\put(20,15){\thicklines\vector(0,-1){12}}

\put(15,1){\line(1,1){31}}

\put(17,0){\thicklines\line(1,1){30}}
\multiput(40,23)(1,1){7}{\line(0,1){1}}

\put(20,15){\thicklines\vector(2,3){5}}

\put(25,22){\thicklines\line(0,1){8}}
\multiput(25,22)(0,1){9}{\line(1,0){1}}

\put(25,22){\thicklines\line(1,0){14}}
\multiput(25,22)(1,0){15}{\line(0,1){1}}

\put(20,15){\thicklines\vector(4,3){24}}

\put(44,33){\line(0,-1){19}}

\put(44,33){\line(-1,0){25}}

\put(0,20){$\bm{a}_{2}=\bm{a}_{2}^{\ast}$}
\put(11,3){$\bm{a}_{1}$}
\put(23,1){$\bm{a}_{1}^{\ast}$}

\put(43,35){$\bm{x}$}

\put(27,25){$\bm{b}$}

\put(43,11){$x_{1}$}

\put(14,33){$x_{2}$}

\put(16,11){$0$}

\end{picture}
\hspace{15\unitlength}
\begin{picture}(50,45)

\put(0,15){\vector(1,0){50}}
\put(20,0){\vector(0,1){45}}

\put(20,15){\thicklines\vector(-1,0){10}}
\put(6,11){\thicklines\line(1,1){30}}
\multiput(22,27)(1,1){14}{\line(1,0){1}}

\put(20,15){\thicklines\vector(-1,-4){3}}

\put(20,15){\thicklines\vector(0,-1){12}}

\put(15,1){\line(1,1){31}}

\put(17,0){\thicklines\line(1,1){30}}
\multiput(44,27)(1,1){3}{\line(0,1){1}}

\put(20,15){\thicklines\vector(-2,3){7.25}}

\put(12.5,26){\line(1,0){30.5}}
\put(21,26){\thicklines\line(1,0){22}}
\multiput(22,26)(1,0){22}{\line(0,1){1}}

\put(20,15){\thicklines\vector(3,4){18}}

\put(38,39){\line(0,-1){25}}

\put(38,39){\line(-1,0){19}}

\put(0,20){$\bm{a}_{2}=\bm{a}_{2}^{\ast}$}

\put(11,3){$\bm{a}_{1}$}
\put(23,1){$\bm{a}_{1}^{\ast}$}

\put(40,38){$\bm{x}$}

\put(10,27){$\bm{b}$}

\put(37,11){$x_{1}$}

\put(14,39){$x_{2}$}

\put(16,11){$0$}

\end{picture}
\end{center}
\caption{Solution to linear inequalities with an irreducible matrix.}
\label{F-GSIAxbx}
\end{figure}
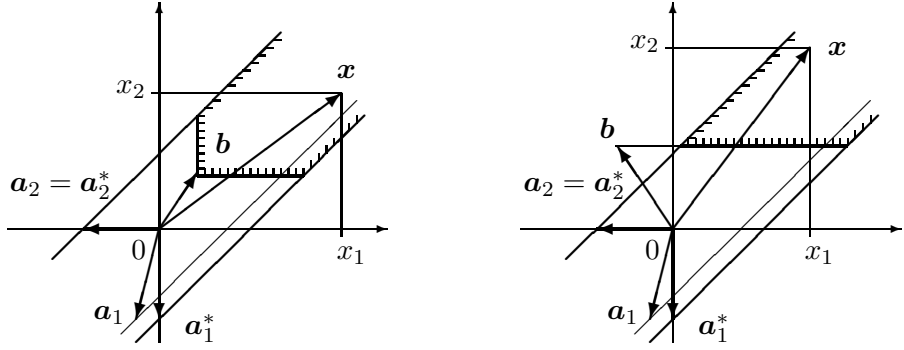

In the rest of this section, we show that the result of Lemma~\ref{L-IAxbx} can be directly extended to solve inequality \eqref{I-Axbx} with irreducible matrices. We start with an auxiliary result that provides a useful representation for $(\bm{A}\oplus\bm{B})^{\ast}$ and can be considered an analogue of an identity that appears in \cite{Carre1979Graphs}.
\begin{proposition}
If $\bm{A}$ and $\bm{B}$ are matrices such that $\mathop\mathrm{Tr}(\bm{A}\oplus\bm{B})\leq\mathbb{1}$, then
\begin{equation}\label{E-ABABA}
(\bm{A}\oplus\bm{B})^{\ast}
=
(\bm{A}^{\ast}\bm{B})^{\ast}\bm{A}^{\ast}.
\end{equation}
\end{proposition}
\begin{proof}
To verify the identity, it is sufficient to show that two opposite inequalities between its left  and right sides are valid. By expanding both the left and the right parts of \eqref{E-ABABA} and rearranging the terms, we obtain one of the inequalities
\begin{multline*}
(\bm{A}\oplus\bm{B})^{\ast}
=
\bigoplus_{m=0}^{n-1}\mathop{\bigoplus\hspace{1.6em}}_{0\leq i_{0}+i_{1}+\cdots+i_{m}\leq n-m-1}\bm{A}^{i_{0}}(\bm{B}\bm{A}^{i_{1}})\cdots(\bm{B}\bm{A}^{i_{m}})
\\
\leq
\bigoplus_{m=0}^{n-1}\mathop{\bigoplus\hspace{0.5em}}_{0\leq i_{0},i_{1},\ldots,i_{m}\leq n-1}\bm{A}^{i_{0}}(\bm{B}\bm{A}^{i_{1}})\cdots(\bm{B}\bm{A}^{i_{m}})
=
(\bm{A}^{\ast}\bm{B})^{\ast}\bm{A}^{\ast}.
\end{multline*}

Furthermore, we denote $\bm{C}=\bm{A}\oplus\bm{B}$ and find that $\bm{A}\leq\bm{C}$ and $\bm{B}\leq\bm{C}$. The application of the Carr{\'e} inequality gives
$$
(\bm{A}^{\ast}\bm{B})^{\ast}\bm{A}^{\ast}
\leq
(\bm{C}^{\ast}\bm{C})^{\ast}\bm{C}^{\ast}
=
\bigoplus_{m=0}^{n^{2}-1}\bm{C}^{m}
=
\bm{C}^{\ast}
=
(\bm{A}\oplus\bm{B})^{\ast},
$$
which shows that the opposite inequality is also true.
\end{proof}

We are now in a position to present the main result of the section.

\begin{theorem}\label{T-Axbx}
Let $\bm{x}$ be the complete regular solution to inequality \eqref{I-Axbx} with an arbitrary matrix $\bm{A}$. Then the following statements hold:
\begin{enumerate}
\item If $\mathop\mathrm{Tr}(\bm{A})\leq\mathbb{1}$, then $\bm{x}=\bm{A}^{\ast}\bm{u}$ for all regular vectors $\bm{u}\geq\bm{b}$.
\item If $\mathop\mathrm{Tr}(\bm{A})>\mathbb{1}$, then there is no regular solution.
\end{enumerate}
\end{theorem}
\begin{proof}
If the matrix $\bm{A}$ is irreducible, then the statement of the theorem is provided by Lemma~\ref{L-IAxbx}. Suppose that $\bm{A}$ is reducible. Provided that $\bm{A}=\mathbb{0}$, the theorem is trivially true. Consider the case in which $\bm{A}\ne\mathbb{0}$ and assume, without loss of generality, that the matrix $\bm{A}$ has the form of \eqref{E-MNF}.

We define block-diagonal and strictly lower block-triangular matrices
$$
\bm{D}
=
\left(
	\begin{array}{ccc}
		\bm{A}_{11}	&					&	\mathbb{0} \\
								& \ddots	& \\
		\mathbb{0}	&					& \bm{A}_{ss}
	\end{array}
\right),
\qquad
\bm{T}
=
\left(
	\begin{array}{cccc}
		\mathbb{0}	& \ldots	& \ldots	& \mathbb{0} \\
		\bm{A}_{21}	& \ddots	&					& \vdots \\
		\vdots			& \ddots	& \ddots	& \vdots \\
		\bm{A}_{s1}	& \ldots	& \bm{A}_{s,s-1}	& \mathbb{0}
	\end{array}
\right)
$$
to be the corresponding parts in the additive decomposition $\bm{A}=\bm{D}\oplus\bm{T}$.

According to the row partitioning of matrix $\bm{A}$ in form \eqref{E-MNF}, we write both vectors $\bm{x}$ and $\bm{b}$ in their block forms:
$$
\bm{x}
=
\left(
	\begin{array}{c}
		\bm{x}_{1} \\
		\vdots \\
		\bm{x}_{s}
	\end{array}
\right),
\qquad
\bm{b}
=
\left(
	\begin{array}{c}
		\bm{b}_{1} \\
		\vdots \\
		\bm{b}_{s}
	\end{array}
\right),
$$
where $\bm{x}_{i}$ and $\bm{b}_{i}$ are vectors of order $n_{i}$ for each $i=1,\ldots,s$.

Furthermore, we represent \eqref{I-Axbx} as a system of inequalities, one for each row block $i$ in the matrix $\bm{A}$:
$$
\bm{A}_{i1}\bm{x}_{1}\oplus\cdots\oplus\bm{A}_{ii}\bm{x}_{i}\oplus\bm{b}_{i}
\leq
\bm{x}_{i},
\qquad
i=1,\ldots,s.
$$

Suppose that $\mathop\mathrm{Tr}(\bm{A})\leq\mathbb{1}$. In this case, for each $i$, we have
$$
\mathop\mathrm{Tr}(\bm{A}_{ii})
\leq
\mathop\mathrm{Tr}(\bm{A}_{11})\oplus\cdots\oplus\mathop\mathrm{Tr}(\bm{A}_{ss})
=
\mathop\mathrm{Tr}(\bm{A})\leq\mathbb{1}.
$$

Based on Lemma~\ref{L-IAxbx}, the inequality for row block $i$ can be solved with respect to $\bm{x}_{i}$. Assuming that $\bm{v}_{i}$ is a regular vector of order $n_{i}$, the solution is given by
$$
\bm{x}_{i}
=
\bm{A}_{ii}^{\ast}\bm{v}_{i},
\qquad
\bm{v}_{i}
\geq
\bm{A}_{i1}\bm{x}_{1}\oplus\cdots\oplus\bm{A}_{i,i-1}\bm{x}_{i}
\oplus
\bm{b}_{i}.
$$

Using another notation $\bm{u}_{i}$ to represent a regular vector of order $n_{i}$, the last inequality can be rewritten as
$$
\bm{v}_{i}
=
\bm{A}_{i1}\bm{x}_{1}\oplus\cdots\oplus\bm{A}_{i,i-1}\bm{x}_{i}
\oplus
\bm{u}_{i},
\qquad
\bm{u}_{i}
\geq
\bm{b}_{i}.
$$

Returning to the solution for row block $i$, we arrive at the representation
$$
\bm{x}_{i}
=
\bm{A}_{ii}^{\ast}(\bm{A}_{i1}\bm{x}_{1}\oplus\cdots\oplus\bm{A}_{i,i-1}\bm{x}_{i})\oplus\bm{A}_{ii}^{\ast}\bm{u}_{i},
\qquad
\bm{u}_{i}
\geq
\bm{b}_{i}.
$$

After the concatenation of all vectors $\bm{u}_{1},\ldots,\bm{u}_{s}$ into one vector $\bm{u}$ of order $n$, we write all of the solutions in the form of an implicit equation for $\bm{x}$:
$$
\bm{x}
=
\bm{D}^{\ast}\bm{T}\bm{x}\oplus\bm{D}^{\ast}\bm{u},
\qquad
\bm{u}
\geq
\bm{b}.
$$

Considering a strictly block-triangular form of $\bm{T}$ and a conforming block-diagonal form of $\bm{D}^{\ast}$, it is not difficult to verify that $(\bm{D}^{\ast}\bm{T})^{r}=\mathbb{0}$ for all $r\geq s$.

By iterating the implicit equation, we arrive at a regular solution
$$
\bm{x}
=
(\bm{I}\oplus(\bm{D}^{\ast}\bm{T})\oplus\cdots\oplus(\bm{D}^{\ast}\bm{T})^{s-1})\bm{D}^{\ast}\bm{u}
=
(\bm{D}^{\ast}\bm{T})^{\ast}\bm{D}^{\ast}\bm{u},
\qquad
\bm{u}
\geq
\bm{b}.
$$

With the application of identity \eqref{E-ABABA} to $\bm{A}=\bm{D}\oplus\bm{T}$, we place the solution in the desired form $\bm{x}=\bm{A}^{\ast}\bm{u}$ with any regular vector $\bm{u}\geq\bm{b}$.

If $\mathop\mathrm{Tr}(\bm{A})>\mathbb{1}$, then the matrix $\bm{A}$ has at least one row block $i$ with $\mathop\mathrm{Tr}(\bm{A}_{ii})>\mathbb{1}$. In this case, according to Lemma~\ref{L-IAxbx}, no regular solutions $\bm{x}_{i}$ exist for the inequality for this row block, which implies that there are no regular solutions $\bm{x}$ to \eqref{I-Axbx}.
\end{proof}

A graphical illustration of the solutions to the inequality with a common reducible matrix $\bm{A}$ in $\mathbb{R}_{\max,+}^{2}$ is given in Figure~\ref{F-GSAxbx}.
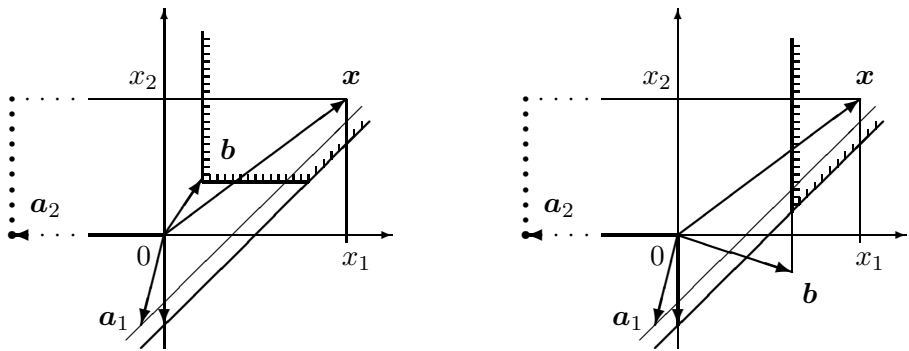
\begin{figure}
\setlength{\unitlength}{1mm}
\begin{center}
\begin{picture}(50,45)

\put(10,15){\vector(1,0){40}}
\put(20,0){\vector(0,1){45}}

\put(0,15){\circle*{1}}
\multiput(0,15)(2,0){5}{\circle*{.5}}

\multiput(0,15)(0,2){10}{\circle*{.75}}

\multiput(0,33)(2,0){5}{\circle*{.5}}

\put(20,15){\thicklines\line(-1,0){10}}

\put(2,15){\thicklines\vector(-1,0){2}}


\put(20,15){\thicklines\vector(-1,-4){3}}

\put(20,15){\thicklines\vector(0,-1){12}}

\put(15,1){\line(1,1){31}}

\put(17,0){\thicklines\line(1,1){30}}
\multiput(40,23)(1,1){7}{\line(0,1){1}}

\put(20,15){\thicklines\vector(2,3){5}}

\put(25,22){\thicklines\line(0,1){20}}
\multiput(25,22)(0,1){20}{\line(1,0){1}}

\put(25,22){\thicklines\line(1,0){14}}
\multiput(25,22)(1,0){15}{\line(0,1){1}}

\put(20,15){\thicklines\vector(4,3){24}}

\put(44,33){\line(0,-1){19}}

\put(44,33){\line(-1,0){34}}

\put(2,18){$\bm{a}_{2}$}

\put(11,3){$\bm{a}_{1}$}

\put(43,35){$\bm{x}$}

\put(27,25){$\bm{b}$}

\put(43,11){$x_{1}$}

\put(15,35){$x_{2}$}

\put(16,11){$0$}

\end{picture}
\hspace{15\unitlength}
\begin{picture}(50,45)

\put(10,15){\vector(1,0){40}}
\put(20,0){\vector(0,1){45}}

\put(0,15){\circle*{1}}
\multiput(0,15)(2,0){5}{\circle*{.5}}

\multiput(0,15)(0,2){10}{\circle*{.75}}

\multiput(0,33)(2,0){5}{\circle*{.5}}

\put(20,15){\thicklines\line(-1,0){10}}

\put(2,15){\thicklines\vector(-1,0){2}}


\put(20,15){\thicklines\vector(-1,-4){3}}

\put(20,15){\thicklines\vector(0,-1){12}}

\put(15,1){\line(1,1){31}}

\put(17,0){\thicklines\line(1,1){30}}
\multiput(36,19)(1,1){11}{\line(0,1){1}}

\put(20,15){\thicklines\vector(3,-1){15}}

\put(35,10){\line(0,1){31}}

\put(35,18){\thicklines\line(0,1){23}}
\multiput(35,19)(0,1){22}{\line(1,0){1}}


\put(20,15){\thicklines\vector(4,3){24}}

\put(44,33){\line(0,-1){19}}

\put(44,33){\line(-1,0){34}}

\put(2,18){$\bm{a}_{2}$}

\put(11,3){$\bm{a}_{1}$}

\put(43,35){$\bm{x}$}

\put(36,6){$\bm{b}$}

\put(43,11){$x_{1}$}

\put(15,35){$x_{2}$}

\put(16,11){$0$}

\end{picture}
\end{center}
\caption{Solution to linear inequalities with a reducible matrix.}
\label{F-GSAxbx}
\end{figure}

\section{A constrained optimisation problem}

We start with a real-world problem drawn from project scheduling and then represent it in terms of the semifield $\mathbb{R}_{\max,+}$. The problem serves as motivation and illustration for the solution of a more common problem that is formulated in the context of a general idempotent semifield, which covers $\mathbb{R}_{\max,+}$ as a specific instance. As the main result, we obtain a solution to the last problem under fairly general assumptions. Particular cases of the problem are considered and illustrated with numerical and graphical examples in the framework of the semifield $\mathbb{R}_{\max,+}$.

\subsection{A project scheduling problem}

We are now concerned with the optimal scheduling of a collection of activities (e.g., jobs and tasks), which presents a rather typical problem in project management (see, e.g., \cite{Institute2010Aguide,Neumann2003Project} for further details and examples).

Consider a project that involves $n$ activities that are conducted under precedence constraints, which are referred to as Start-to-Finish, Start-to-Start, and Early Start constraints. The Start-to-Finish constraints do not allow an activity to be completed until predefined times after the initiation of other activities. The activities are assumed to be completed as early as possible to meet the constraints. The Start-to-Start constraints determine the minimum allowed times between the initiations of activities, whereas the Early Start constraints specify the earliest possible dates for the initiations of activities. 

Each activity in a schedule is characterised by its flow time (also known as turnaround or processing time), which is defined as the time interval between its initiation and its completion. The problem is to design a schedule that minimises the maximum flow time over all activities subject to the above mentioned precedence constraints.

For each activity $i=1,\ldots,n$, we denote its initiation time by $x_{i}$ and its completion time by $y_{i}$. Let $a_{ij}$ be the minimum possible time lag between the initiation of activity $j=1,\ldots,n$ and the completion of activity $i$.

Given $a_{ij}$, the completion time of activity $i$, due to the Start-to-Finish constraints, must satisfy the relations
$$
x_{j}+a_{ij}\leq y_{i},
\quad
j=1,\ldots,n,
$$ 
where at least one inequality holds as an equality. Note that, in the case when $a_{ij}$ is not given for some $j$, it is assumed to be $-\infty$.

The relations can be combined to obtain one obvious equality of the form
$$
y_{i}
=
\max(x_{1}+a_{i1},\ldots,x_{n}+a_{in}).
$$

By replacing the ordinary operations with those in $\mathbb{R}_{\max,+}$, we obtain
$$
y_{i}
=
a_{i1}x_{1}\oplus\cdots\oplus a_{in}x_{n},
\quad
i=1,\ldots,n.
$$

Using a matrix $\bm{A}=(a_{ij})$ and vectors $\bm{x}=(x_{i})$ and $\bm{y}=(y_{i})$, we represent the scalar equalities in a vector form
$$
\bm{y}
=
\bm{A}\bm{x}.
$$

Furthermore, we formulate a problem for the minimisation of the maximum flow time of the activities. In ordinary notation, the objective function in the problem is given by
$$
\max(y_{1}-x_{1},\ldots,y_{n}-x_{n}).
$$

Since, in terms of $\mathbb{R}_{\max,+}$, the function is written as
$$
x_{1}^{-1}y_{1}\oplus\cdots\oplus x_{n}^{-1}y_{n}
=
\bm{x}^{-}\bm{y}
=
\bm{x}^{-}\bm{A}\bm{x},
$$
we obtain a problem that requires us to find all regular $\bm{x}$ such that
\begin{equation*}
\begin{aligned}
&
\text{minimise}
&&
\bm{x}^{-}\bm{A}\bm{x}.
\end{aligned}
\end{equation*}

We now incorporate the Start-to-Start and Early Start constraints in the problem formulation. For each activity $i=1,\ldots,n$, we use $g_{i}$ to denote the earliest possible initiation time and $c_{ij}$ to denote the minimum possible time lag between the initiation of activity $j=1,\ldots,n$ and the initiation of activity $i$. Given $c_{ij}$ and $g_{i}$, the initiation time $x_{i}$ for activity $i$ is subject to the relations
$$
\max(x_{1}+c_{i1},\ldots,x_{n}+c_{in})
\leq
x_{i},
\qquad
g_{i}
\leq
x_{i}.
$$

The representation in terms of $\mathbb{R}_{\max,+}$ results in scalar inequalities:
$$
c_{i1}x_{1}\oplus\cdots\oplus c_{in}x_{n}
\leq
x_{i},
\qquad
g_{i}
\leq
x_{i}.
$$

Using the matrix-vector notation $\bm{C}=(c_{ij})$ and $\bm{g}=(g_{i})$, we obtain
$$
\bm{C}\bm{x}
\leq
\bm{x},
\qquad
\bm{g}
\leq
\bm{x}.
$$

The optimal scheduling problem under consideration takes the form
\begin{equation*}
\begin{aligned}
&
\text{minimise}
&&
\bm{x}^{-}\bm{A}\bm{x},
\\
&
\text{subject to}
&&
\bm{C}\bm{x}
\leq
\bm{x},
\quad
\bm{g}
\leq
\bm{x}.
\end{aligned}
\end{equation*}

In conclusion, we note that the solution to the problem (if any exists) is not unique. Specifically, if there is a solution $\bm{x}$, then $\alpha\bm{x}$, where $\alpha>\mathbb{1}=0$, is also a solution. In the context of project scheduling, however, a unique solution naturally appears when the minimal solution is of interest. 

\subsection{A general representation}

Let $\mathbb{X}$ be a linearly ordered radicable idempotent semifield. Given two matrices $\bm{A},\bm{C}\in\mathbb{X}^{n\times n}$ and a vector $\bm{g}\in\mathbb{X}^{n}$, the problem is to find all regular solutions $\bm{x}\in\mathbb{X}_{+}^{n}$ that
\begin{equation}
\begin{aligned}
&
\text{minimise}
&&
\bm{x}^{-}\bm{A}\bm{x},
\\
&
\text{subject to}
&&
\bm{C}\bm{x}
\leq
\bm{x},
\quad
\bm{g}
\leq
\bm{x}.
\end{aligned}
\label{P-xAxCxxxg}
\end{equation}

Since, for all vectors $\bm{x}$, the inequalities $\bm{C}\bm{x}\leq\bm{x}$ and $\bm{x}\geq\bm{g}$ can be readily combined into one equivalent inequality $\bm{C}\bm{x}\oplus\bm{g}\leq\bm{x}$, the problem has an alternative representation in the form
\begin{equation}
\begin{aligned}
&
\text{minimise}
&&
\bm{x}^{-}\bm{A}\bm{x},
\\
&
\text{subject to}
&&
\bm{C}\bm{x}\oplus\bm{g}
\leq
\bm{x}.
\end{aligned}
\label{P-xAxCxxxg1}
\end{equation}

Consider the set of regular vectors that are determined by the inequality constraints in the problem. An appropriate necessary and sufficient condition for the set to be nonempty follows from Theorem~\ref{T-Axbx} and requires that $\mathop\mathrm{Tr}(\bm{C})\leq\mathbb{1}$.

\subsection{The main result}

A complete solution to the optimisation problem under study is now given by an expression to represent a general vector that describes all regular vectors in the solution set, and to be called below the general regular solution. 

\begin{theorem}\label{T-xAxCxxxg}
Let $\bm{x}$ be the general regular solution to problem \eqref{P-xAxCxxxg}, where $\bm{A}$ is a matrix with spectral radius $\lambda>\mathbb{0}$ and $\bm{C}$ is a matrix with $\mathop\mathrm{Tr}(\bm{C})\leq\mathbb{1}$, and let
\begin{equation}
\theta
=
\lambda
+
\bigoplus_{k=1}^{n-1}\mathop{\bigoplus\hspace{1.1em}}_{1\leq i_{1}+\cdots+i_{k}\leq n-k}\mathop\mathrm{tr}\nolimits^{1/k}(\bm{A}\bm{C}^{i_{1}}\cdots\bm{A}\bm{C}^{i_{k}}).
\label{E-theta}
\end{equation}

Then, the minimum in \eqref{P-xAxCxxxg} is equal to $\theta$ and
\begin{equation}
\bm{x}
=
(\theta^{-1}\bm{A}\oplus\bm{C})^{\ast}\bm{u},
\qquad
\bm{u}
\geq
\bm{g}.
\label{E-xSu}
\end{equation}
\end{theorem}
\begin{proof}
We start with the problem in the form of \eqref{P-xAxCxxxg1} and then follow the approach suggested in \cite{Krivulin2012Atropical}. We introduce an additional variable and reduce the problem to an inequality. Furthermore, existence conditions for solutions to the inequality are used to determine the variable, whereas the solution to the inequality is considered as the solution to the original problem. 

Suppose that $\theta$ is the minimum of the objective function in \eqref{P-xAxCxxxg1}. Note that $\theta\geq\lambda>\mathbb{0}$. All regular vectors $\bm{x}$ that yield the minimum are therefore given by the system
$$
\bm{x}^{-}\bm{A}\bm{x}
=
\theta,
\qquad
\bm{C}\bm{x}\oplus\bm{g}
\leq
\bm{x}.
$$

Consider the first equality $\bm{x}^{-}\bm{A}\bm{x}=\theta$. Since, for all $\bm{x}$, we actually have $\bm{x}^{-}\bm{A}\bm{x}\geq\theta$, the equality can be further replaced with the inequality $\bm{x}^{-}\bm{A}\bm{x}\leq\theta$.

We then take the inequality $\bm{x}^{-}\bm{A}\bm{x}\leq\theta$ and multiply it by $\theta^{-1}\bm{x}$ from the left. Since $\bm{x}\bm{x}^{-}\geq\bm{I}$ for all regular $\bm{x}$, we have $\theta^{-1}\bm{A}\bm{x}\leq\theta^{-1}\bm{x}\bm{x}^{-}\bm{A}\bm{x}\leq\bm{x}$ and thus obtain a new inequality $\theta^{-1}\bm{A}\bm{x}\leq\bm{x}$.

Considering that the left multiplication of the last inequality by $\theta\bm{x}^{-}$ and an obvious equality $\bm{x}^{-}\bm{x}=\mathbb{1}$ give the former one, both inequalities are equivalent.

The solution set to the problem is now defined by the system of inequalities
$$
\theta^{-1}\bm{A}\bm{x}
\leq
\bm{x},
\qquad
\bm{C}\bm{x}\oplus\bm{g}
\leq
\bm{x}.
$$

It is easy to verify that the system is equivalent to one inequality:
\begin{equation}
(\theta^{-1}\bm{A}\oplus\bm{C})\bm{x}\oplus\bm{g}
\leq
\bm{x}.
\label{I-thetaACxgx}
\end{equation}

It follows from Theorem~\ref{T-Axbx} that a necessary and sufficient condition for regular solutions to the inequality to exist is given by
\begin{equation}
\mathop\mathrm{Tr}(\theta^{-1}\bm{A}\oplus\bm{C})
\leq
\mathbb{1}.
\label{I-Trtheta1AC1}
\end{equation}

To further examine the existence condition, we rewrite the left side of the inequality and then apply the binomial identity \eqref{E-trABm} to obtain
\begin{multline*}
\mathop\mathrm{Tr}(\theta^{-1}\bm{A}\oplus\bm{C})
=
\bigoplus_{m=1}^{n}\mathop\mathrm{tr}(\theta^{-1}\bm{A}\oplus\bm{C})^{m}
\\
=
\bigoplus_{m=1}^{n}\bigoplus_{k=1}^{m}\mathop{\bigoplus\hspace{2.3em}}_{i_{1}+\cdots+i_{k}=m-k}\mathop\mathrm{tr}(\theta^{-1}\bm{A}\bm{C}^{i_{1}}\cdots\theta^{-1}\bm{A}\bm{C}^{i_{k}})
\oplus
\bigoplus_{m=1}^{n}\mathop\mathrm{tr}(\bm{C}^{m})
\\
=
\bigoplus_{m=1}^{n}\bigoplus_{k=1}^{m}\mathop{\bigoplus\hspace{2.3em}}_{i_{1}+\cdots+i_{k}=m-k}\theta^{-k}\mathop\mathrm{tr}(\bm{A}\bm{C}^{i_{1}}\cdots\bm{A}\bm{C}^{i_{k}})
\oplus
\mathop\mathrm{Tr}(\bm{C}).
\end{multline*}

Furthermore, the rearrangement of the terms in the sum gives the following expression:
$$
\mathop\mathrm{Tr}(\theta^{-1}\bm{A}\oplus\bm{C})
=
\bigoplus_{k=1}^{n}\mathop{\bigoplus\hspace{1.2em}}_{0\leq i_{1}+\cdots+i_{k}\leq n-k}\theta^{-k}\mathop\mathrm{tr}(\bm{A}\bm{C}^{i_{1}}\cdots\bm{A}\bm{C}^{i_{k}})
\oplus
\mathop\mathrm{Tr}(\bm{C}).
$$

Considering that $\mathop\mathrm{Tr}(\bm{C})\leq\mathbb{1}$ based on the conditions of the theorem, the existence condition \eqref{I-Trtheta1AC1} reduces to the inequalities
$$
\mathop{\bigoplus\hspace{1.2em}}_{0\leq i_{1}+\cdots+i_{k}\leq n-k}\theta^{-k}\mathop\mathrm{tr}(\bm{A}\bm{C}^{i_{1}}\cdots\bm{A}\bm{C}^{i_{k}})
\leq
\mathbb{1},
\qquad
k=1,\ldots,n.
$$

By solving the inequalities with respect to $\theta$, we obtain the inequality
\begin{multline*}
\theta
\geq
\bigoplus_{k=1}^{n}\mathop{\bigoplus\hspace{1.2em}}_{0\leq i_{1}+\cdots+i_{k}\leq n-k}\mathop\mathrm{tr}\nolimits^{1/k}(\bm{A}\bm{C}^{i_{1}}\cdots\bm{A}\bm{C}^{i_{k}})
\\
=
\lambda
\oplus
\bigoplus_{k=1}^{n-1}\mathop{\bigoplus\hspace{1.2em}}_{1\leq i_{1}+\cdots+i_{k}\leq n-k}\mathop\mathrm{tr}\nolimits^{1/k}(\bm{A}\bm{C}^{i_{1}}\cdots\bm{A}\bm{C}^{i_{k}}).
\end{multline*}

For $\theta$ to be the minimum in the problem, the last inequality must be satisfied as an equality, which gives \eqref{E-theta}.

Finally, the application of Theorem~\ref{T-Axbx} to inequality \eqref{I-thetaACxgx} leads to the solution in the form of \eqref{E-xSu}.
\end{proof}

Note that the minimum solution to the problem is given by $\bm{x}_{0}=(\theta^{-1}\bm{A}\oplus\bm{C})^{\ast}\bm{g}$.

\subsection{Particular cases}

First, assume that $\bm{C}=\mathbb{0}$ and consider the following problem:
\begin{equation}
\begin{aligned}
&
\text{minimise}
&&
\bm{x}^{-}\bm{A}\bm{x},
\\
&
\text{subject to}
&&
\bm{x}
\geq
\bm{g}.
\end{aligned}
\label{P-xAxxg}
\end{equation}

The next solution to the problem is a direct consequence of Theorem~\ref{T-xAxCxxxg}.

\begin{corollary}\label{C-xAxxg}
For any matrix $\bm{A}$ with spectral radius $\lambda>\mathbb{0}$, the minimum in problem \eqref{P-xAxxg} is $\lambda$ and the general regular solution to the problem is given by
$$
\bm{x}
=
(\lambda^{-1}\bm{A})^{\ast}\bm{u},
\qquad
\bm{u}
\geq
\bm{g}.
$$
\end{corollary}

Now consider the case in which the problem has no lower bound constraints and is thus given by 
\begin{equation}
\begin{aligned}
&
\text{minimise}
&&
\bm{x}^{-}\bm{A}\bm{x},
\\
&
\text{subject to}
&&
\bm{C}\bm{x}
\leq
\bm{x}.
\end{aligned}
\label{P-xAxCxx}
\end{equation}

By setting $\bm{g}=\mathbb{0}$ in Theorem~\ref{T-xAxCxxxg}, we immediately arrive at the following result.
\begin{corollary}\label{C-xAxCxx}
For any matrix $\bm{A}$ with spectral radius $\lambda>\mathbb{0}$ and matrix $\bm{C}$ with $\mathop\mathrm{Tr}(\bm{C})\leq\mathbb{1}$, the minimum in problem \eqref{P-xAxCxx} is defined by \eqref{E-theta} and the general regular solution to the problem is given by
$$
\bm{x}
=
(\theta^{-1}\bm{A}\oplus\bm{C})^{\ast}\bm{u},
\qquad
\bm{u}
>
\mathbb{0}.
$$
\end{corollary}

It is easy to see that, when at least one of the matrices $\bm{A}$ and $\bm{C}$ is irreducible, the last result coincides with that in \cite{Krivulin2012Atropical}.

Finally, for a problem without any linear constraint in the form
\begin{equation}
\begin{aligned}
&
\text{minimise}
&&
\bm{x}^{-}\bm{A}\bm{x},
\end{aligned}
\label{P-xAx}
\end{equation}
the solution offered by the theorem is reduced as follows.
\begin{corollary}\label{C-xAx}
For any matrix $\bm{A}$ with spectral radius $\lambda>\mathbb{0}$, the minimum in problem \eqref{P-xAx} is $\lambda$ and the general regular solution to the problem is given by
$$
\bm{x}
=
(\lambda^{-1}\bm{A})^{\ast}\bm{u},
\qquad
\bm{u}
>
\mathbb{0}.
$$
\end{corollary}

Note that this result is quite consistent with that in \cite{Krivulin2012Acomplete}.

\subsection{Illustrative examples}

To illustrate the results obtained, we provide numerical and graphical examples in $\mathbb{R}_{\max,+}^{2}$ for problem \eqref{P-xAxCxxxg} with both irreducible and reducible matrices.

\subsubsection{Irreducible matrices}

We start with a problem with irreducible matrices $\bm{A}$ and $\bm{C}$, which are defined, in addition to vector $\bm{g}$, as follows:
$$
\bm{A}
=
\left(
\begin{array}{rr}
0 & -2
\\
-7 & -3
\end{array}
\right),
\qquad
\bm{C}
=
\left(
\begin{array}{rr}
0 & -10
\\
4 & -3
\end{array}
\right),
\qquad
\bm{g}
=
\left(
\begin{array}{r}
-9
\\
6
\end{array}
\right).
$$

First, consider problem~\eqref{P-xAx}, which has no constraints. Taking into account the matrices $\bm{A}$ and
$$
\bm{A}^{2}
=
\left(
\begin{array}{rr}
0 & -2
\\
-7 & -3
\end{array}
\right)
\left(
\begin{array}{rr}
0 & -2
\\
-7 & -3
\end{array}
\right)
=
\left(
\begin{array}{rr}
0 & -2
\\
-7 & -6
\end{array}
\right),
$$
we find the spectral radius $\lambda=\mathop\mathrm{tr}\bm{A}\oplus\mathop\mathrm{tr}\nolimits^{1/2}(\bm{A}^{2})=0=\mathbb{1}$. Then, we get
$$
(\lambda^{-1}\bm{A})^{\ast}
=
\bm{A}^{\ast}
=
\bm{I}\oplus\bm{A}
=
\left(
\begin{array}{rr}
0 & -2
\\
-7 & 0
\end{array}
\right).
$$ 

The solution to the unconstrained problem takes the form
$$
\bm{x}
=
\left(
\begin{array}{rr}
0 & -2
\\
-7 & 0
\end{array}
\right)
\bm{u},
\qquad
\bm{u}\in\mathbb{R}^{2},
$$
which is graphically represented as in Figure~\ref{F-IUPCP} (left).

Furthermore, we obtain the solution to the inequality constraints in the problem. To apply Theorem~\ref{T-Axbx}, we take the matrices $\bm{C}$ and
$$
\bm{C}^{2}
=
\left(
\begin{array}{rr}
0 & -10
\\
4 & -3
\end{array}
\right)
\left(
\begin{array}{rr}
0 & -10
\\
4 & -3
\end{array}
\right)
=
\left(
\begin{array}{rr}
0 & -10
\\
4 & -6
\end{array}
\right).
$$

We then verify that $\mathop\mathrm{Tr}(\bm{C})=\mathop\mathrm{tr}\bm{C}\oplus\mathop\mathrm{tr}\bm{C}^{2}=0=\mathbb{1}$. After calculating
$$
\bm{C}^{\ast}
=
\bm{I}
\oplus
\bm{C}
=
\left(
\begin{array}{rr}
0 & -10
\\
4 & 0
\end{array}
\right)
$$
we arrive at the solution (see Figure~\ref{F-IUPCP}, right)
$$
\bm{x}
=
\left(
\begin{array}{rr}
0 & -10
\\
4 & 0
\end{array}
\right)
\bm{u},
\qquad
\bm{u}
\geq
\left(
\begin{array}{r}
-9
\\
6
\end{array}
\right).
$$

We now use Theorem~\ref{T-xAxCxxxg} to solve the constrained problem. We calculate
$$
\bm{A}\bm{C}
=
\left(
\begin{array}{rr}
0 & -2
\\
-7 & -3
\end{array}
\right)
\left(
\begin{array}{rr}
0 & -10
\\
4 & -3
\end{array}
\right)
=
\left(
\begin{array}{rr}
2 & -5
\\
1 & -6
\end{array}
\right),
$$
and then get $\theta=\lambda\oplus\mathop\mathrm{tr}(\bm{A}\bm{C})=2$. Furthermore, we have
$$
\theta^{-1}\bm{A}\oplus\bm{C}
=
\left(
\begin{array}{rr}
0 & -4
\\
4 & -3
\end{array}
\right),
\qquad
(\theta^{-1}\bm{A}\oplus\bm{C})^{\ast}
=
\bm{I}\oplus\theta^{-1}\bm{A}\oplus\bm{C}
=
\left(
\begin{array}{rr}
0 & -4
\\
4 & 0
\end{array}
\right).
$$

All solutions to the problem take the form
$$
\bm{x}
=
\left(
\begin{array}{rr}
0 & -4
\\
4 & 0
\end{array}
\right)\bm{u},
\qquad
\bm{u}
\geq
\left(
\begin{array}{r}
-9
\\
6
\end{array}
\right).
$$

To simplify the solution, note that the columns in matrix $(\theta^{-1}\bm{A}\oplus\bm{C})^{\ast}$ are collinear. Therefore, we can write
$$
\bm{x}
=
(\theta^{-1}\bm{A}\oplus\bm{C})^{\ast}\bm{u}
=
\left(
\begin{array}{rr}
0 & -4
\\
4 & 0
\end{array}
\right)
\bm{u}
=
\left(
\begin{array}{r}
0
\\
4
\end{array}
\right)
\left(
\begin{array}{rr}
0 & -4
\end{array}
\right)
\bm{u}.
$$

Using a new scalar variable
$$
v
=
\left(
\begin{array}{rr}
0 & -4
\end{array}
\right)
\bm{u}
\geq
\left(
\begin{array}{rr}
0 & -4
\end{array}
\right)
\bm{g}
=
\left(
\begin{array}{rr}
0 & -4
\end{array}
\right)
\left(
\begin{array}{r}
-9
\\
6
\end{array}
\right)
=
2,
$$
the solution to problem~\eqref{P-xAxCxxxg} gets its final form
$$
\bm{x}
=
\left(
\begin{array}{r}
0
\\
4
\end{array}
\right)
v,
\qquad
v\geq
2.
$$

Figure~\ref{F-IUPCP} (right) combines the solution constraints with the final solution to the entire constrained problem. The final solution is depicted with a thick half-line that coincides with the right border of the feasible area defined by the constraints. The half-line starts at the end point of vector $\bm{x}_{0}$, which represents the minimal solution.
\begin{figure}
\setlength{\unitlength}{1mm}
\begin{center}
\begin{picture}(45,45)

\put(0,20){\vector(1,0){45}}
\put(20,0){\vector(0,1){45}}

\put(20,20){\thicklines\vector(-2,-3){4.5}}
\put(5,10){\thicklines\line(1,1){26}}
\multiput(6,11)(1,1){25}{\line(1,0){1}}

\put(20,20){\thicklines\vector(-1,0){5}}

\put(9,7){\line(1,1){26}}

\put(20,20){\thicklines\vector(0,-1){14}}
\put(16,2){\thicklines\line(1,1){26}}
\multiput(17,3)(1,1){25}{\line(-1,0){1}}

\put(11,23){$\bm{a}_{2}^{\ast}$}

\put(15,10){$\bm{a}_{2}$}

\put(22,4){$\bm{a}_{1}=\bm{a}_{1}^{\ast}$}

\end{picture}
\hspace{20\unitlength}
\begin{picture}(50,45)

\put(0,20){\vector(1,0){50}}
\put(25,0){\vector(0,1){45}}

\put(25,20){\thicklines\vector(-2,-3){4.5}}
\put(10,10){\line(1,1){26}}
\multiput(11,11)(1,1){24}{\line(1,0){1}}



\put(25,20){\thicklines\vector(0,-1){14}}
\put(21,2){\line(1,1){26}}
\multiput(22,3)(1,1){25}{\line(-1,0){1}}

\put(25,20){\thicklines\vector(-4,-1){20}}
\put(2,17){\line(1,1){26}}
\multiput(17,32)(1,1){11}{\line(1,0){1}}

\put(25,20){\thicklines\vector(-1,0){20}}


\put(25,20){\thicklines\vector(0,1){9}}
\put(8,12){\line(1,1){26}}
\multiput(28,32)(1,1){6}{\line(0,1){1}}
\put(28,32){\pmb{\thicklines\line(1,1){6}}}

\put(25,20){\thicklines\vector(-3,2){18}}
\put(7,32){\line(1,0){21}}
\multiput(18,32)(1,0){11}{\line(0,1){1}}

\put(25,20){\thicklines\vector(1,4){3}}


\put(26,36){$\bm{x}_{0}$}

\put(4,33){$\bm{g}$}

\put(18,10){$\bm{a}_{2}$}

\put(1,22){$\bm{c}_{2}^{\ast}$}

\put(1,12){$\bm{c}_{2}$}

\put(27,4){$\bm{a}_{1}$}

\put(28,24){$\bm{c}_{1}=\bm{c}_{1}^{\ast}$}

\end{picture}
\end{center}
\caption{Solutions to unconstrained (left) and constrained (right) problems with irreducible matrices.}
\label{F-IUPCP}
\end{figure}
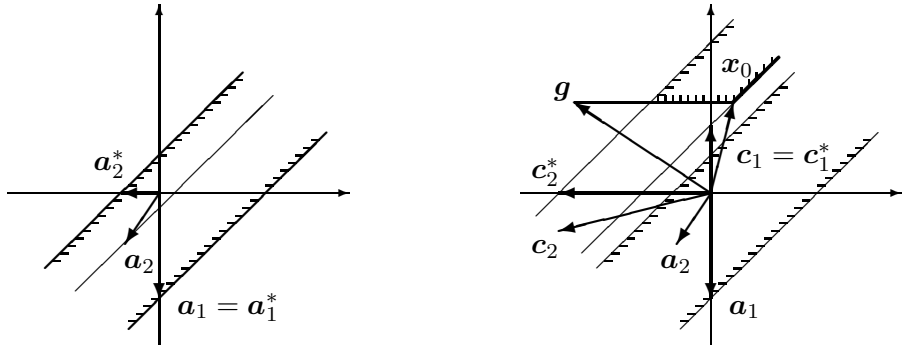

\subsubsection{Reducible matrices}

We now solve problem~\eqref{P-xAxCxxxg} with reducible matrices under the conditions that
$$
\bm{A}
=
\left(
\begin{array}{rr}
-2 & -\infty
\\
-4 & 0
\end{array}
\right),
\qquad
\bm{C}
=
\left(
\begin{array}{rr}
0 & -6
\\
-\infty & -4
\end{array}
\right),
\qquad
\bm{g}
=
\left(
\begin{array}{r}
3
\\
4
\end{array}
\right).
$$

We successively calculate
$$
\bm{A}^{2}
=
\left(
\begin{array}{rr}
-4 & -\infty
\\
-4 & 0
\end{array}
\right),
\qquad
\lambda
=
0,
\qquad
(\lambda^{-1}\bm{A})^{\ast}
=
\bm{A}^{\ast}
=
\left(
\begin{array}{rr}
0 & -\infty
\\
-4 & 0
\end{array}
\right).
$$

The solution to the problem without constraints is given by (see Figure~\ref{F-UPCP}, left)
$$
\bm{x}
=
\left(
\begin{array}{rr}
0 & -\infty
\\
-4 & 0
\end{array}
\right)
\bm{u},
\qquad
\bm{u}\in\mathbb{R}^{2}.
$$

Furthermore, we obtain the solution to the inequality constraints. Since
$$
\bm{C}^{2}
=
\left(
\begin{array}{rr}
0 & -6
\\
-\infty & -8
\end{array}
\right),
\qquad
\mathop\mathrm{Tr}(\bm{C})
=
0,
\qquad
\bm{C}^{\ast}
=
\left(
\begin{array}{rr}
0 & -6
\\
-\infty & 0
\end{array}
\right),
$$
the solution takes the form (Figure~\ref{F-UPCP}, right)
$$
\bm{x}
=
\left(
\begin{array}{rr}
0 & -6
\\
-\infty & 0
\end{array}
\right),
\qquad
\bm{u}
\geq
\left(
\begin{array}{r}
3
\\
4
\end{array}
\right).
$$

To solve the constrained problem, we apply Theorem~\ref{T-xAxCxxxg}. We have
$$
\bm{A}\bm{C}
=
\left(
\begin{array}{rr}
-2 & -8
\\
-4 & -4
\end{array}
\right),
\qquad
\theta
=
0,
\qquad
\theta^{-1}\bm{A}\oplus\bm{C}
=
(\theta^{-1}\bm{A}\oplus\bm{C})^{\ast}
=
\left(
\begin{array}{rr}
0 & -6
\\
-4 & 0
\end{array}
\right),
$$
and then write the solution to problem~\eqref{P-xAxCxxxg} in the form
$$
\bm{x}
=
\left(
\begin{array}{rr}
0 & -6
\\
-4 & 0
\end{array}
\right)\bm{u},
\qquad
\bm{u}
\geq
\left(
\begin{array}{r}
3
\\
4
\end{array}
\right).
$$

The solution is shown in Figure~\ref{F-UPCP} (right) as a half-strip region surrounded by a thick hatched border.
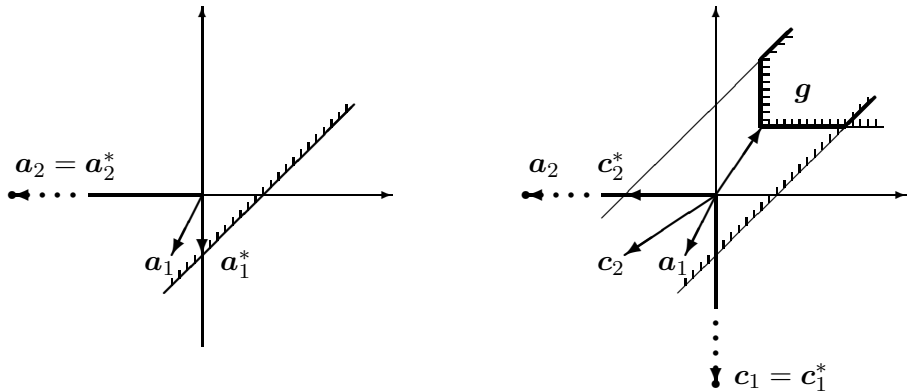
\begin{figure}
\setlength{\unitlength}{1mm}
\begin{center}
\begin{picture}(50,50)

\put(10,25){\vector(1,0){40}}
\put(25,5){\vector(0,1){45}}

\put(0,25){\circle*{1}}
\multiput(0,25)(2,0){5}{\circle*{.75}}

\put(25,25){\thicklines\line(-1,0){15}}
\put(2,25){\thicklines\vector(-1,0){2}}



\put(25,25){\thicklines\vector(0,-1){8}}
\put(25,25){\thicklines\vector(-1,-2){4}}


\put(20,12){\thicklines\line(1,1){25}}
\multiput(21,13)(1,1){24}{\line(0,1){1}}





\put(0,28){$\bm{a}_{2}=\bm{a}_{2}^{\ast}$}

\put(17,15){$\bm{a}_{1}$}
\put(27,15){$\bm{a}_{1}^{\ast}$}


\end{picture}
\hspace{15\unitlength}
\begin{picture}(50,50)

\put(10,25){\vector(1,0){40}}
\put(25,10){\vector(0,1){40}}

\put(0,25){\circle*{1}}
\multiput(0,25)(2,0){5}{\circle*{.75}}

\put(25,25){\thicklines\line(-1,0){15}}
\put(2,25){\thicklines\vector(-1,0){2}}

\put(25,0){\circle*{1}}
\multiput(25,0)(0,2){5}{\circle*{.75}}

\put(25,25){\thicklines\line(0,-1){15}}
\put(25,2){\thicklines\vector(0,-1){2}}

\put(25,25){\thicklines\vector(-1,-2){4}}

\put(31,43){\pmb{\thicklines\line(1,1){4}}}
\put(10,22){\line(1,1){25}}
\multiput(31,43)(1,1){4}{\line(1,0){1}}

\put(42,34){\pmb{\thicklines\line(1,1){4}}}
\put(20,12){\line(1,1){25}}
\multiput(21,13)(1,1){25}{\line(0,1){1}}

\put(25,25){\thicklines\vector(-3,-2){12}}
\put(25,25){\thicklines\vector(-1,0){12}}

\put(25,25){\thicklines\vector(2,3){6}}

\put(31,34){\pmb{\thicklines\line(0,1){9}}}
\multiput(31,34)(0,1){9}{\line(1,0){1}}

\put(31,34){\line(1,0){16}}
\put(31,34){\pmb{\thicklines\line(1,0){11}}}
\multiput(31,34)(1,0){16}{\line(0,1){1}}

\put(0,28){$\bm{a}_{2}$}

\put(17,15){$\bm{a}_{1}$}

\put(9,28){$\bm{c}_{2}^{\ast}$}
\put(9,15){$\bm{c}_{2}$}

\put(27,0){$\bm{c}_{1}=\bm{c}_{1}^{\ast}$}

\put(35,38){$\bm{g}$}

\end{picture}
\end{center}
\caption{Solutions to unconstrained (left) and constrained (right) problems with reducible matrices.}
\label{F-UPCP}
\end{figure}

\section{Conclusions}

We started this paper with an overview of known tropical optimisation problems to demonstrate that tropical optimisation is currently under intense development as a promising solution approach to many real-world problems. The paper examines a new optimisation problem with a nonlinear objective function and linear inequality constraints in a rather general setting. The problem sufficiently extends that described in \cite{Krivulin2012Atropical} by eliminating the irreducibility conditions for the matrices involved and by introducing additional constraints. We offer an example problem drawn from project scheduling as motivation and illustration of the study.

The derivation of the solution is based on the definitions, notations, and basic results developed in \cite{Krivulin2006Solution,Krivulin2006Eigenvalues,Krivulin2009Methods}, which are outlined in a brief introduction that is supplemented by appropriate graphical illustrations. As a preliminary step, a new complete solution to a linear inequality was obtained for the case of arbitrary matrices. The solution extends the known results presented in \cite{Krivulin2006Solution,Krivulin2009Methods} and is of independent interest.

The solution to the optimisation problem under study follows the approach suggested in \cite{Krivulin2012Atropical}, which reduces the problem to the solution of a linear inequality with a parameterised matrix. By applying the above result for linear inequalities, we obtain a complete solution to the optimisation problem in a compact vector form. To illustrate the result, numerical solutions and graphical representations are given for two-dimensional problems.

Future research is expected to concentrate on the further extension of the problem to account for new types of constraints, including equality constraints, and a more general form of the objective function. Another line of investigation will be the development of new real-world applications of the results.

\section*{Acknowledgements}

This work was supported in part by the Russian Foundation for Humanities (grant No.~13-02-00338). The author thanks two referees and the associate editor for valuable comments and suggestions, which have been incorporated into the revised version of the manuscript.

\bibliographystyle{abbrvurl}

\bibliography{A_multidimensional_tropical_optimization_problem_with_a_nonlinear_objective_function_and_linear_constraints}

\end{document}